\documentclass{amsproc}

\usepackage{amssymb}

\usepackage[arrow,matrix,cmtip,all]{xy}

\usepackage{url}
\usepackage[mathscr]{eucal}
\usepackage{hyperref}

\let\wt=\widetilde \let\:=\colon \let\ox=\otimes
\let\?=\overline \let\ts=\textstyle
\def\tq#1{\text{\quad#1\quad}}
 \newcommand{\CC}{\mathbb{C}}

\newcommand{\PP}{\mathbb{P}}
 \newcommand{\VV}{\mathbb{V}}
 \newcommand{\ZZ}{\mathbb{Z}}

\newcommand{\Ll}{\mathbf{L}}

\newcommand{\Tt}{\mathbf{T}}
\def\Cc{\mathbf{C}}

\let\mathcal=\mathscr
\newcommand{\aA}{\mathcal{A}}
\newcommand{\cC}{\mathcal{C}}
\newcommand{\fF}{\mathcal{F}}
\newcommand{\jJ}{\mathcal{J}}
\newcommand{\hH}{\mathcal{H}}
 \newcommand{\lL}{\mathcal{L}}
 
\newcommand{\nN}{\mathcal{N}}
 \newcommand{\oO}{\mathcal{O}}
 \newcommand{\rR}{\mathcal{R}}

\def\cln{\smash{\cC_{|\lL|}^{[n]}}}

\def\(#1){{\rm(#1)}}
\let\into=\hookrightarrow
\let\onto=\twoheadrightarrow
\let\xto\xrightarrow
\def\Into{\lhook\joinrel\longrightarrow}
\def\smashedlongrightarrow{\setbox0=\hbox{$\longrightarrow$}\ht0=1pt\box0}
\def\risom{\buildrel\sim\over{\smashedlongrightarrow}}

\def\Dfl{\mathop{{\rm Def}_{\rm loc}}}
\def\Def{\mathop{\rm Def}}

\def\Ker{\mathop{\rm Ker}} \def\Im{\mathop{\rm Im}}
\def\Coker{\mathop{\rm Coker}}
\def\Hom{\mathcal{H}om}
\def\red{{\mathop{\rm red}}}
\def\Spec{\mathop{\rm Spec}}
\def\Supp{\mathop{\rm Supp}}

   \DeclareMathOperator\mult{mult}
 \DeclareMathOperator\rHom{Hom}  \DeclareMathOperator\Ext{Ext}
\DeclareMathOperator{\ecH}{H}
\DeclareMathOperator{\codim}{codim}

\newtheorem{lemma}{Lemma}
\newtheorem{proposition}[lemma]{Proposition}

\newtheorem{corollary}[lemma]{Corollary}
\newtheorem{theorem}[lemma]{Theorem}

\theoremstyle{definition}
 \newtheorem{remark}[lemma]{Remark}

\numberwithin{equation}{lemma}

\begin{document}

\title{On the G\"ottsche Threshold}


\author[S. Kleiman]{Steven L. Kleiman}
 \address
 {Dept. of Math., 2-278 MIT\\
 77 Mass. Ave.\\
 Cambridge, MA 02139, USA}
 \email{Kleiman@math.MIT.edu}

\author[V. Shende]{Vivek V. Shende}
 \address
 {Dept. of Math., MIT 2-248\\
 77 Mass. Ave.\\
 Cambridge, MA 02139, USA}
 \email{Vivek@math.MIT.edu}
 \thanks{Vivek Shende gratefully acknowledges the support of the Simons
Foundation.}

\author[with an appendix by I. Tyomkin]{\\with an appendix by Ilya Tyomkin}
 \address
 {Department of Mathematics, Ben-Gurion University of the Negev, P.O.Box
653, Be'er Sheva, 84105, Israel}
 \email{tyomkin@cs.bgu.ac.il}
\thanks{Ilya Tyomkin gratefully acknowledges the support of the European
Union Seventh Framework Programme (FP7/2007-2013) under
grant agreement 248826}



\date{}

\begin{abstract}
For a line bundle $\lL$ on a smooth surface $S$, it is now known that
the degree of the Severi variety of cogenus-$\delta$ curves is given by
a universal polynomial in the Chern classes of $\lL$ and $S$ if $\lL$ is
$\delta$-very ample.  For $S$ rational, we relax the latter condition
substantially: it suffices that three key loci be of codimension more
than $\delta$.  As corollaries, we prove that the condition conjectured
by G\"ottsche suffices if $S$ is $\PP^2$ or $S$ is any Hirzebruch
surface, and that a similar condition suffices if $S$ is any classical
del Pezzo surface.
\end{abstract}

\maketitle

\section{Introduction}\label{Intro}
Fix $\delta\ge0$.  Fix a smooth irreducible projective complex surface
$S$, and a line bundle $\lL$.  Let $|\lL|$ be the complete linear
system, and $|\lL|^{\delta}\subset |\lL|$ the Severi variety, the locus
of reduced curves $C$ of \textit{cogenus} $\delta$; so $\delta$ is the
genus drop, $\delta:=p_aC-p_gC$, or $\delta=\chi(\oO_{\wt
C})-\chi(\oO_C)$ where $\wt C$ is the normalization.  Let
$|\lL|^{\delta}_+\subset |\lL|^\delta$ be the sublocus of $\delta$-nodal
curves.  Often enough when $S$ is rational, $|\lL|^\delta_+$ is open and
dense in $|\lL|^\delta$, so that $\deg |\lL|^{\delta}_+ = \deg
|\lL|^{\delta}$; see Prp.~\ref{prNC} below.

The degree $\deg|\lL|^\delta_+$ can be found recursively if $S$ is the
plane \cite[Thm.\,3C.1]{R}, \cite[Thm.\,1.1]{CH}, if $S$ is any
Hirzebruch (rational ruled) surface \cite[\S\,8]{Va}, or if $S$ is any
classical del Pezzo surface (that is, its anticanonical bundle is very
ample) \cite[\S\,9]{Va}.  If $\delta$ and $S$ are arbitrary, but $\lL$ is
sufficiently ample, then by \cite{L,L04}, by \cite{T}, or by \cite{KST},
there's a universal polynomial $G_\delta(S,L)$ in the Chern classes of
$S$ and $\lL$ with
\begin{equation}\label{eq1}
 \deg |\lL|^{\delta}_+ = G_\delta(S,\lL).\tag{$+$}
\end{equation}

Further, set $r:=\dim|\lL|$.  In those cases, $\deg|\lL|^\delta_+$ is
the number of $\delta$-nodal curves through $r-\delta$ general points,
and each curve is counted with multiplicity 1 by \cite[Lem.\,(4.7)]{KP}.
See \cite{K11} for a brief survey of related work and open problems.

{\em Given $\delta$ and $S$, for precisely which $\lL$ does
\eqref{eq1} hold?}  It is known \cite[Thm.\,4.1]{KST} that
\eqref{eq1} holds if $\lL$ is $\delta$-very ample, that is if, for any
subscheme $Z \subset S$ of length $\delta+1$, the restriction map
$\ecH^0(\lL)\to \ecH^0(\lL|_Z)$ is surjective.  In particular, \eqref{eq1}
holds for $S=\PP^2$ and $\lL=\oO(d)$ if $d\ge\delta$.  Previously, this
bound had been confirmed by F. Block \cite[Thm.\,1.3]{Bl}, who also
coined the term {\it G\"ottsche threshold\/} for the value of $d$ at
which \eqref{eq1} begins to hold.  However, as conjectured by G\"ottsche
\cite[Cnj.\,4.1, Rmk.\,4.4]{G} and proved by Block \cite[Thm.\,1.4]{Bl}
for $\delta =3,\dotsc,14$, in fact the threshold appears to be
$\lceil\delta/2\rceil+1$ if $\delta\ge3$; whereas, it is $1$ if
$\delta=0,1,2$.  G\"ottsche \cite[Rmk.\,4.3, 4.4]{G} also conjectured a
value for the threshold if $S$ is any Hirzebruch surface.

Here we prove G\"ottsche's conjectured value is at least an upper bound on
the threshold if $S$ is $\PP^2$ or if $S$ is any Hirzebruch surface, and
we prove a similar bound if $S$ is any classical del
 Pezzo surface; see Cors.\,\ref{coP2},\,\ref{coH},\,\ref{codP} and
Rmk.\,\ref{reH} stated just below.  Although we cannot say exactly when
the bound is tight, in Rmk.\,\ref{reH} we show it isn't if $S$ is the
first Hirzebruch surface, the blowup of $\PP^2$ at a point.  We derive
those results directly from our main results, Thm.\,\ref{ththreshold}
and Prp.\,\ref{prNC}, stated next.

Note that the term immersed is used here in the sense of differential
geometry; specifically, we call an embedded curve $D\subset S$ {\it
immersed\/} if $D$ is reduced and the tangent map $T_{\wt D}\to T_S$ is
injective, where $\wt D$ is the normalization.

\begin{theorem}\label{ththreshold}
 Assume $S$ is rational with canonical class $K$.  Let $V$ be a closed
subset of $|\lL|$ that contains every $D\in |\lL|$ such that either
\begin{enumerate}
 \item  $D$ is nonreduced, or
 \item $D$ has a component $D_1$ with $-K\cdot D_1\le 0$, or
 \item $D$ has a nonimmersed component $D_1$ with $-K\cdot D_1= 1$.
 \end{enumerate} Then the closure of\/ $|\lL|^\delta-V$ has codimension
$\delta$ at all its points (if any), and its sublocus of immersed curves
is open and dense, and is smooth off\/ $V$.  Further, if $\codim V
>\delta$, then $|\lL|^\delta$ has codimension $\delta$ at all its
points, and $\deg |\lL|^{\delta} = G_\delta(S,\lL)$.  \end{theorem}

\begin{proposition}\label{prNC} Under the conditions of
Thm.\,{\rm\ref{ththreshold}}, assume $D\in V$ also if either
\begin{enumerate}\setcounter{enumi}{3}
 \item $D$ has a component $D_1$ with a point of multiplicity at least
$3$ and with $-K\cdot D_1\le 3$, or
  \item $D$ has two components $D_1,\, D_2$ with a common point that is
double on $D_1$ and with $-K\cdot D_1=1$ or $-K\cdot D_2=1$, or
  \item $D$ has two components $D_1,\, D_2$ with a common point that is
double on $D_1$ and on $D_2$ and with $-K\cdot D_1=2$ and $-K\cdot
D_2=2$, or
  \item $D$ has two components $D_1,\, D_2$ with a common point that is
double on  $D_1$  and simple on $D_2$ and with $-K\cdot D_1=2$, or
 \item $D$ has three components $D_1,\, D_2,\, D_3$ with a common point
that is simple on each and with $-K\cdot D_1=1$, or
 \item $D$ has two components $D_1,\, D_2$ with a common point that is simple on
each and at which they are tangent and with  $-K\cdot D_1=1$, or
 \item $D$ has a component $D_1$ with a nonnodal double point and with
$-K\cdot D_1\le 2$.
 \end{enumerate}
Then in the closure of\/ $|\lL|^\delta-V$, its sublocus of nodal curves
is open and dense.  Further, if $\codim V >\delta$, then
$|\lL|^\delta_+$ is open and dense in $|\lL|^\delta$, and \eqref{eq1}
holds.
\end{proposition}

\begin{corollary} \label{coP2}
Assume $S=\PP^2$ and $\lL=\oO(d)$.  If $d\ge \lceil\delta/2\rceil+1$, then
\eqref{eq1} holds.
\end{corollary}

\begin{corollary} \label{coH}
Assume $S$ is the Hirzebruch surface with section $E$ of
self-intersection $-e$ with $e\ge0$.  Assume these
subloci of $|\lL|$ have codimension more than $\delta$: \(1) the
nonreduced curves, \(2) if $e \ge 1$, the curves with $E$ as a
component.  Then \eqref{eq1} holds.
\end{corollary}

\begin{remark}\label{reH}
 G\"ottsche \cite[Rmk.\,4.3, 4.4]{G} stated without proof that the
codimension condition of Cor.\,\ref{coH} is equivalent to essentially
this condition:
 say $\lL=\oO(nF+mE)$ where $F$ is a ruling, and set $p:=n-em$; then
either $m=0$, $p=1$, and $\delta=1$ or
\begin{equation}\label{eqreH}
m+p\ge1
\tq{and}
\delta\le\begin{cases} \min(2m,\,p) &\text{if }e\ge1,\\
                        \min(2m,\,2p)  &\text{if }e=0.
          \end{cases}
\end{equation}

In fact, more is true; the proof of this equivalence plus the main
results yield the following statements.  Assume $e\ge1$ and $m\ge2$ and
$p\ge0$.  Assume the nonreduced $D\in|\lL|$ appear in codimension more
than $\delta$, or equivalently,
\begin{equation}\label{eqnonred}
 \delta\le\min(2m,\,2p+e+1).
\end{equation}
Assume $\delta\ge p+e$ too.  Then there
are curves in $|\lL|^\delta$ with $E$ as a component, and they form a
component of $|\lL|^\delta$ of codimension $\delta-e+1$; the other
components are of codimension $\delta$.  Lastly, if $e=1$, then $\deg
|\lL|^{\delta} = G_\delta(S,\lL)$; further, \eqref{eq1} holds at least
if $\delta=p+1$ too.
\end{remark}

\begin{corollary}\label{codP}
Assume $S$ is a classical del Pezzo surface.  Assume these subloci of
$|\lL|$ have codimension more than $\delta$: \(1) the nonreduced curves,
\(2) the curves with a $-1$-curve as a component.  Then \eqref{eq1}
holds.
\end{corollary}

Section 2 derives the three corollaries from the theorem and the
proposition.  It also proves the remark.  Section 3 proves four lemmas
about the Severi variety and the relative Hilbert scheme.  Section 4
uses those lemmas to prove the theorem and the proposition, which are
the main results.

Throughout, $\delta$, $S$, $\lL$, $K$, and so forth continue to be as
above.  In particular, $C$ denotes a reduced member of $|\lL|$, and $D$
an arbitrary member.  In addition, $\Gamma$ denotes an arbitrary reduced
curve on $S$, usually integral, but not always.

As some loci may be empty, we adopt the convention that the empty set
has dimension $-1$, and so codimension 1 more than the dimension of the
ambient space.  Thus, in the theorem and the proposition, the hypothesis
$\codim V >\delta\ge0$ implies that $\dim|\lL|\ge0$; in particular,
$\lL$ is nontrivial.

\section{Proof of the corollaries and the remark}\label{Tsrs}
 Before addressing the corollaries and the remark, we prove the following
lemma, which we use to handle the bounds in Cor.\,\ref{coP2} and
Rmk.\,\ref{reH}.

\begin{lemma}\label{lecdB} Assume that $S$ is rational and that
$D\in|\lL|$.  Then $\ecH^2(S,\,\lL)=0$ and $\dim |\lL| \ge
D\cdot(D-K)/2$.  Equality holds and $\ecH^1(S,\,\lL)=0$ if this
condition obtains:
 every component $\Gamma\!$ of $D$ satisfies $-K\cdot \Gamma\ge1$, and
every $\Gamma$ that is a $-1$-curve appears with multiplicity $1$.
\end{lemma}
\begin{proof}
 Since $S$ is integral, $\ecH^0(S,\,\oO_S)=1$.  Since $S$ is rational,
$\ecH^q(S,\,\oO_S)=0$ for $q=1,2$.  Hence the Riemann--Roch theorem
yields
\begin{equation}\label{eqRR}
\dim |\lL| = D\cdot(D-K)/2 +\dim\ecH^1(S,\,\lL)-\dim\ecH^2(S,\,\lL).
\end{equation}
Thus it suffices to study the vanishing of $H^1(S,\,\lL)$ and
$H^2(S,\,\lL)$.

Given a component $\Gamma$ of $D$, let $m_\Gamma$ denote its
multiplicity of appearance.  Set $m:=\sum m_\Gamma$, and proceed by
induction on $m$.  Suppose $m=0$. Then $D=0$.  So $\lL=\oO_S$.  Hence in
this case, both groups vanish.

Suppose $m\ge1$.  Fix a component $\Gamma$, and set
$\lL':=\lL(-\Gamma)$.  Form the standard sequence
$0\to\lL'\to\lL\to\lL|\Gamma\to0$, and
take cohomology to get this sequence:
\begin{equation*}\label{eqstiso}
\ecH^q(S,\,\lL')\to\ecH^q(S,\,\lL)\to\ecH^q(S,\,\lL|\Gamma)\tq{for} q=1,2.
\end{equation*}
 By induction, $\ecH^2(S,\,\lL')=0$.  As $\Gamma$ is a curve,
$\ecH^2(\Gamma,\,\lL|\Gamma)=0$.  Thus $\ecH^2(S,\,\lL)=0$, as desired.

Assume the stated condition obtains.  Then by induction,
$\ecH^1(S,\,\lL')=0$.  Thus, it suffices to show
$\ecH^1(\Gamma,\,\lL|\Gamma)=0$.

Let $K_\Gamma$ be the canonical class.  By adjunction,
$\oO_\Gamma(K_\Gamma) = \oO_\Gamma(\Gamma+K)$.  So
$$\ecH^1(\Gamma,\,\lL|\Gamma)
  =\ecH^1(\Gamma,\,\oO_\Gamma(D-\Gamma+K_\Gamma-K)).$$
The latter group is dual to $\ecH^0(\Gamma,\,\oO_\Gamma(-D+\Gamma+K))$,
which vanishes as desired, since $\Gamma$ is integral and since, as
shown next, $(-D+\Gamma+K)\cdot\Gamma<0$.

First, by hypothesis, $K\cdot\Gamma<0$.  Second, if $m_\Gamma=1$, then
$D-\Gamma$ does not contain $\Gamma$, and so
$(-D+\Gamma)\cdot\Gamma\le0$.  Finally, suppose $m_\Gamma\ge2$.  Then,
by hypothesis, $\Gamma$ is not a $-1$-curve; so $\Gamma^2\neq-1$ if
$K\cdot\Gamma=-1$.  But $(\Gamma+K)\cdot\Gamma=\deg K_\Gamma\ge-2$.  So
$\Gamma^2\ge-K\cdot\Gamma-2\ge-1$.  Hence $\Gamma^2\ge0$.  Thus again
$(-D+\Gamma)\cdot\Gamma\le0$, as desired.
\end{proof}

Note in passing that, if $\lL=\oO_S(m\Gamma)$ where $\Gamma$ is a
$-1$-curve and $m\ge1$, then \eqref{eqRR} yields
$\dim\ecH^1(S,\,\lL)=m(m-1)/2$.

\begin{proof}[{\scshape Proof of Cor.\,\ref{coP2}}] Note $\deg K=-3$; so
$-K\cdot \Gamma\ge 3$ for every integral curve $\Gamma$ on $S$, and
$-K\cdot \Gamma\ge 9$ if $\Gamma$ is singular.  So no $D\in|\lL|$
satisfies any of (2)--(10) of Thm.\,\ref{ththreshold} and
Prp.\,\ref{prNC}.  Thus it remains to consider (1).

The nonreduced $D\in|\lL|$ are of the form $D=A+2B$ with $A,\,B$
effective.  Set $b:=\deg B$.  Fix $b\ge1$.  Then these $D$ form a locus
of dimension $\dim|A|+\dim|B|$, so of codimension $b(4d-5b+3)/2$ owing
to Lem.\,\ref{lecdB}.  But $d\ge2b$.  So
\begin{align*}\label{alP2}
b(4d-5b+3)/2-(2d-1) &= (b-1)(4d-5b-2)/2\\
 &\ge (b - 1)(3 b - 2)/2\ge0.
\end{align*}
Therefore, when $b=1$, the codimension achieves its minimum value, namely,
$2d-1$.  This value is more than $\delta$, as desired.
\end{proof}

\begin{proof}[{\scshape Proof of Cor.\,\ref{coH}}] For the following basic
properties of Hirzebruch surfaces, see \cite[Ch.\,V, \S\,2]{H}.  Let $F$
be a ruling.  Then every curve $\Gamma$ is equivalent to $nF+mE$
with $n,\, m \ge 0$.  Suppose $\Gamma$ is integral and $\Gamma\ne E$.
Then $n>0$ and $n-me \ge 0$.  Further, $-K = (e+2)F + 2E$.  Finally,
$F^2=0$ and $F\cdot E=1$.

Hence $-K\cdot \Gamma = n + (n-me) + 2m$.  Suppose $-K\cdot \Gamma\le3$.
Then either $n=1$ and $m=0$, or $n,\,m,\,e=1$.  In first case, $-K\cdot
\Gamma=2$; further, $\Gamma=F$, so $\Gamma$ is smooth.  In the second
case, $-K\cdot \Gamma=3$; further, $\Gamma\cdot F=1$, whence $\Gamma$ is
smooth.  On the other hand, $E$ is smooth, and $-K\cdot E=2-e$.  So if
$-K\cdot E\le1$, then $e\ge1$.

In $|\lL|$ consider the locus of $D$ with a component $\Gamma$ such that
$-K\cdot\Gamma\le k$.  By the above, if $k=1$, then $\Gamma=E$ and
$e\ge1$.  So by hypothesis, the locus has codimension more than
$\delta$.  Further, if $k=3$, then $\Gamma$ is smooth.  Thus all the
hypotheses of Thm.\,\ref{ththreshold} and Prp.\,\ref{prNC} obtain;
whence, \eqref{eq1} holds, as asserted.
\end{proof}

\begin{proof}[{\scshape Proof of Rmk.\,\ref{reH}}] Fix a section $G$ of $S$
complementary to $E$.  Then $G$ is equivalent to $eF+E$, so that
$\lL=\oO(pF+mG)$.  Let's see that, if there's a $D\in|\lL|$, then
$m\ge0$; further, $p\ge0$ if also either $e=0$ or $e\ge1$ and $D$
doesn't contain $E$.  Indeed, as $|F|$ has no base points, $m=D\cdot
F\ge0$.  If $e=0$, then $S=\PP^1\times\PP^1$; whence by symmetry,
$p\ge0$.  If $e\ge1$, then $p=D\cdot E\ge0$.

Note that, if the nonreduced $D\in|\lL|$ form a locus of codimension
more than $\delta$, then $\dim|\lL|\ge0$; in particular,
$\lL$ is nontrivial.  Then $m\ge0$.  Further, if some
$D\in|\lL|$ doesn't contain $E$, then $p=D\cdot E\ge0$.  In particular,
if the codimension condition of Cor.\,\ref{coH} obtains, then $m,\,p\ge0$.
On the other hand, if \eqref{eqreH} obtains, then $m,\,p\ge\delta\ge0$.
Thus to prove the remark, we may assume $m,\,p\ge0$ and
$m+p\ge1$.

If $m=0$ and $p=1$, then $\dim|\lL|=1$, no $D\in|\lL|$ contains $E$, and
every $D$ is reduced; whence, then the codimension condition of
Cor.\,\ref{coH} obtains if and only if $\delta\le1$, if and only if
either $\delta=1$ or \eqref{eqreH} obtains.  If $m=0$ and $p\ge2$, then
$\dim|\lL|\ge 2$, no $D\in|\lL|$ contains $E$, and the nonreduced $D$
form a locus of codimension 1.  Hence, then the codimension condition of
Cor.\,\ref{coH} obtains if and only if $\delta=0$, if and only if
\eqref{eqreH} obtains.  Thus, to complete the proof, we may assume
$m\ge1$; further, if $e=0$, then by symmetry, we may assume $p\ge1$ too.

The proof of Cor.\,\ref{coH} yields $-K\cdot F=2$ and $-K\cdot G=e+2$.
Also $\lL=\oO(pF+mG)$ and $m,\,p\ge0$.  So Lem.\,\ref{lecdB} yields this
formula:
 \begin{equation*}\label{eqL} \dim|\lL|=p m + p + m +m e (1 +
m)/2.  \end{equation*}

The $D\in|\lL|$ containing $E$ are of the form $D=A+E$ with $A$
effective.  Set
\begin{equation}\label{eqL'}
\lL':=\oO_S\bigl((p+e)F+(m-1)G\bigr).
\end{equation}
Then $A\in|\lL'|$.  But we now assume $p\ge0$ and $m\ge1$.  So
Lem.\,\ref{lecdB} yields
 \begin{equation*}\label{eqdmeA}
  \dim|\lL'|=p m -1 + m +m e (1 + m)/2\ge1.
 \end{equation*}

If $e\ge1$, then $\dim|E|=0$ as $E^2=-e$ (whereas if $e=0$, then
$\dim|E|=1$); so the $D\in|\lL|$ containing $E$ form a nonempty
locus of codimension exactly $p=1$:
 \begin{equation*}\label{eqcdeE} \dim|\lL|-\dim|\lL'| = p+1.
\end{equation*}
 Thus, if $e\ge1$, then the $D\in|\lL|$ containing $E$ appear in
codimension more than $\delta$ if and only if $\delta\le p$.

By the same token, if $e\ge1$ and if $m\ge2$, then the $A\in|\lL'|$
containing $E$ appear in codimension $p+e+1$.  Conversely, if $e\ge1$
and if there exists such an $A$, then $m-2=(A-E)\cdot F\ge0$.  Thus if
$e\ge1$, then there exists a $D\in|\lL|$ containing $2E$ if and only if
$m\ge2$; if so, then these $D$ form a locus of codimension $2p+e+2$.

Given a nonreduced $D\in|\lL|$, say $D=A+2B$ with $A,B$ effective and
$B\neq0$.  Say $B$ is equivalent to $aF+bG$.  Then $A$ is equivalent to
$(p-2a)F+(m-2b)G$.  Since $A$ and $B$ are effective, $m-2b\ge0$ and
$b\ge0$.  If $e\ge1$, assume $D$ does not contain $E$.  Then $p-2a\ge0$
and $a\ge0$ for any $e$.  Hence, for fixed $a$ and $b$, these $D$ form a
locus of dimension $\dim|A|+\dim|B|$; so Lem.\,\ref{lecdB} yields its
codimension to be
\begin{equation*}\label{eqdmeA2B}
\epsilon(a,b):=2 p b + 2 a m - 5 a b + a + b + (1 + 4m - 5b)be/2.
\end{equation*}

The above analysis assumed given some $D$ and $A$ and $B$.  However,
given $a,\,b\ge0$ such that  $p-2a\ge0$ and $m-2b\ge0$, set
 $$A:=(p-2a)F+(m-2b)G,\quad B:=aF+bG,\quad D:=A+2B.$$
Then $A$ and $B$ are effective.  Also, $D\in|\lL|$, and $D$ does not
contain $E$.  Further, $B\neq0$ if $a+b\ge1$.  So the above
analysis yields a locus of nonreduced members of $|\lL|$ of codimension
$\epsilon(a,b)$.

Note $\epsilon(0,1)=2p+1+2e(m-1)$.  But $p\ge2a$ and $m\ge2b$.  So if
$b\ge1$, then
\begin{align*}\label{aldme01}
\epsilon(a,b)-\epsilon(0,1) &= (2p+1)(b-1)+a(2 m - 5b+1) \\
  &\qquad\quad+(4m-5b-4)(b-1)e/2\\
 &\ge (3a+1 +(3b-4)e/2)(b-1).
\end{align*}
The latter term is nonnegative.  Further,
$$\epsilon(a,0)=a(2m+1)\ge \epsilon(1,0)=2m+1.$$
Thus $\min\epsilon(a,b)=\min\bigl(\epsilon(1,0), \,\epsilon(0,1)\bigr)
=\min\bigl(2m+1,\,2p+1+2e(m-1)\bigr)$.

Suppose $e=0$.  Then we are assuming $m,\,p\ge1$.  Hence the nonreduced
$D\in|\lL|$ form a nonempty locus of codimension exactly
$\min\bigl(2m+1,\,2p+1\bigr)$.  Thus the codimension condition of
Cor.\,\ref{coH} obtains if and only if \eqref{eqreH} obtains, as
asserted.

Suppose $e\ge1$ and the codimension condition of Cor.\,\ref{coH}
obtains.  In this case, we assume $m\ge1$ and $p\ge0$.  Then, as proved
above, $\delta\le p$.  So if $p\le 1$, then $\delta\le 2m$.  If $p\ge2$,
take $a:=1$ and $b:=0$; then the codimension condition yields
$\epsilon(1,0)>\delta$.  But $\epsilon(1,0)=2m+1$.  Thus \eqref{eqreH}
obtains, as asserted.

Conversely, suppose $e\ge1$ and \eqref{eqreH} obtains.  Then, as proved
above, the $D\in|\lL|$ containing $E$ appear in codimension more than
$\delta$.  Also, the nonreduced $D\in|\lL|$ not containing $E$ appear in
codimension $\min\bigl(2m+1,\,2p+1+2e(m-1)\bigr)$.  But we assume
$m-1\ge0$.  Thus the codimension condition of Cor.\,\ref{coH} obtains,
as asserted.

Finally, assume $e\ge1$ and $m\ge2$.  Then $2e(m-1)\ge e+1$.  Let $W$ be
the locus of all nonreduced curves.  Then $\codim
W=\min(2m+1,\,2p+e+2)$.  Thus $\codim W>\delta$ if and only if
\eqref{eqnonred} obtains, as asserted.  Assume \eqref{eqnonred} does obtain.

Assume $\delta\ge p+e$ too.  Set $\delta':=\delta-p-e$.  Then
$\delta'\le p+1$ as $\delta\le2p+e+1$; so $\delta'\le p+e$ as $e\ge1$.
Further, $\delta\le 2m$, so $\delta'\le 2m-p-e$.  Hence
$\delta'\le2m-2$, except possibly if $p=0$; but then, $\delta'\le1$, so
after all $\delta'\le2m-2$ as $m\ge2$.

Consider the $\lL'$ of \eqref{eqL'}.  By the above analysis, the
Severi variety $|\lL'|^{\delta'}$ is nonempty and everywhere of
codimension $\delta'$ in $|\lL'|$, so of codimension $\delta-e+1$ in
$|\lL|$.  Further, $|\lL'|^{\delta'}$ contains a dense open subset of
curves $A$ not containing $E$.  Set $D:=A+E$.  Then $D\in|\lL|^\delta$
as $p_aD=p_aA+p_aE+A\cdot E-1$ and $p_gD=p_gA+p_gE-1$ by general
principles.  Conversely, given a $D\in|\lL|^\delta$ containing $E$, set
$A:=D-E$; then, plainly, $A\in|\lL'|^{\delta'}$, and $A$ does not
contain $E$.

Recall that $\codim W>\delta$; further, if $\Gamma$ is an
integral curve with $-K\cdot\Gamma\le1$, then $\Gamma=E$.  Let $V$ be
the union of $W$ and the locus of $D\in|\lL|$ containing $E$.  Then by
Thm.\,\ref{ththreshold}, the closure of $|\lL|^\delta-V$ has
codimension $\delta$ everywhere.  Consequently, there are
$D\in|\lL|^\delta$ containing $E$, and they form a component of
$|\lL|^\delta$ of codimension $\delta-e+1$; the other components of
$|\lL|^\delta$ are of codimension $\delta$, as asserted.

Lastly, assume $e=1$ in addition.  Then $-K\cdot E=1$ and $E$ is
immersed.  Thus Thm.\,\ref{ththreshold} yields $\deg |\lL|^{\delta} =
G_\delta(S,\lL)$, as asserted.

Further, by Prp.\,\ref{prNC}, the nodal curves form an open and dense
subset of $|\lL|^\delta-V$.  Assume $\delta=p+1$ also.  Then
$\delta'=0$.  So the $D\in|\lL|^\delta$ containing $E$ are the
$D\in|\lL|$ of the form $A+E$ where $A\in|\lL'|-V$.  The $A$ that meet
$E$ transversally form a dense open sublocus, because the restriction
map $\ecH^0(S,\,\lL)\to\ecH^0(S,\,\lL|E)$ is surjective as
$\ecH^1(S,\,\lL)=0$ by Lem.\, \ref{lecdB}.  Hence the nodal locus is
open and dense in
$|\lL|^\delta$.  Thus \eqref{eq1}
holds, as asserted.
\end{proof}

\begin{proof}[{\scshape Proof of Cor.\,\ref{codP}}] Since $S$ is a classical
del Pezzo surface, we may regard $S$ as embedded in a projective
space with $-K$ as the hyperplane class.  Let $\Gamma\subset S$ be an
integral curve.  Suppose $-K\cdot \Gamma =1$.  Then $\Gamma$ is a line.
So  adjunction yields $\Gamma^2=-1$.  Hence $\Gamma$ is a
$-1$-curve.  In $|\lL|$ consider the locus of $D$ with a component $D_1$
such that $-K\cdot D_1=1$; by hypothesis, this locus therefore has
codimension more than $\delta$.  If $-K\cdot \Gamma =2$, then $\Gamma$
is an integral plane conic, so smooth.  Finally, if $-K\cdot \Gamma =3$,
then $\Gamma$ is either a twisted cubic, so smooth, or else an integral
plane cubic, so has no point of multiplicity at least 3.  Thus all the
hypotheses of Thm.\,\ref{ththreshold} and Prp.\,\ref{prNC} obtain;
whence, \eqref{eq1} holds, as asserted.
\end{proof}

\section{Four lemmas}\label{Fl}
We now set the stage to prove Thm.\,\ref{ththreshold} and
Prp.\,\ref{prNC}.  First off, we recall some basic deformation theory
from \cite{DH} and \cite{GLS}.

Fix the reduced curve $C\in|\lL|$.  There exist a smooth (analytic or
\'etale) germ
$$(\Lambda,0):=(\Dfl(C),\,0)$$
 and a family $\cC_\Lambda\big/ \Lambda$ realizing a miniversal
deformation of the singularities of $C$; that is, given any family
$\cC_B \big/ B$ and point $b\in B$ such that the fiber $\cC_b$ is a
multigerm of $C$ along its singular locus $\Sigma$, there exists a map of
germs $(B,b) \to (\Lambda,0)$ such that the multigerm $(\cC_B,\Sigma)$
is the pullback of the multigerm $(\cC_\Lambda,\Sigma)$.  The tangent
map $T_b B \to T_0 \Lambda$ is canonical.  Further, there is an
identification
 \begin{equation}\label{eqtnsp}
  T_0\Lambda = \ecH^0(C,\,\oO_C / \jJ)
 \end{equation}
 where $\jJ$ is the {\it Jacobian\/} ideal of $C$, the first Fitting ideal
of its K\"ahler differentials.

Denote the cogenus of $C$  by $\delta(C)$ and the {\it normalization\/}  map by
$$n\:\wt C\to C.$$
So $\delta(C)=\dim\ecH^0 (n_*\oO_{\wt C}\big/\oO_C)$.  Denote the locus
of $a\in\Lambda$ with $\delta(\cC_a) = \delta(C)$ by
$\Lambda^{\delta(C)}$.  It is called the {\em equigeneric locus} or {\em
$\delta$-constant stratum}.  Its codimension is $\delta(C)$.  Its
reduced tangent cone $(\Cc_0\Lambda^\delta)_\red$ is a vector space;
namely,
\begin{equation}\label{eqredtgtcone}
(\Cc_0\Lambda^\delta)_\red = \ecH^0(C,\,\aA/\jJ)
\end{equation}
under the identification \eqref{eqtnsp}.  Here $\aA$ denotes the {\it
conductor\/} ideal sheaf; namely,
\begin{equation*}\label{eqcndtr}
 \aA:=\Hom(n_*\oO_{\wt{C}},\,\oO_C).
\end{equation*}

The following lemma regarding $\aA$ is fundamental.  It is more or less
well known.

\begin{lemma} \label{leVan}
Denote by $K_{\wt C}$ the canonical class of $\wt C$.
Then
\begin{equation}\label{eqZar}
 \aA\cdot n^*\oO_S(C) = O_{\wt C}(K_{\wt C}-n^*K)
\advance \belowdisplayskip by -3pt
\end{equation}
where, doing double duty, $n$ also denotes the composition $n\:\wt C\to
C\into S$.

 Let $\wt M$ be a line bundle on $\wt C$, and $\wt C_1,\dotsc,\wt C_h$
be the components of $\wt C$.  Then
\begin{equation}\label{eqVan}
  \ts
 \dim\ecH^1\big(C,\ \aA\cdot n_*\wt M\ox\oO_S(C)\bigr) \le
  \sum_{i=1}^h\max\bigl(0,\
        1+\deg\bigl(\wt M^{-1}(n^*K)\big|\wt C_i\bigr)\bigr).
\end{equation}
\end{lemma}

\begin{proof}
By adjunction, $\oO_C(K_C) = \oO_C\ox\oO_S(C+K)$.  And relative duality yields
$$n_*\oO_{\wt C}(K_{\wt C})
  = \Hom(n_*\oO_{\wt{C}},\,\oO_C(K_C))
  = \aA\ox \oO_C(K_C).$$
Hence $\aA\ox\oO_S(C)=n_*\oO_{\wt C}(K_{\wt C})\ox\oO_S(-K)$.  But $n$
is finite, and that equation is just the image under $n_*$ of
\eqref{eqZar}.  Thus \eqref{eqZar} holds.

By the same token, $\ecH^1(C,\,\aA\cdot n_*\wt M\ox\oO_S(C)) =
 \ecH^1(\wt C,\,\wt M(K_{\wt C}-n^*K))$.  By duality, the right side is
just $\smash{\ecH^0(\wt C,\,\wt M^{-1}(n^*K))^\vee}$; whence,
\eqref{eqVan} holds.
\end{proof}

Since $C\in |\lL|$, the tangent map $T_C|\lL| \to T_0\Lambda$ is just
this restriction map:
 \begin{equation}\label{eqRes}
\ecH^0(S,\,\lL)\big/\Im \ecH^0(S,\oO_S) \to \ecH^0(C,\,\oO_C / \jJ).
 \end{equation} Consequently, using Lem.\,\ref{leVan}, we can prove
the following results about the Severi variety and the Hilbert scheme.
The results about the Severi variety are already known in various forms,
see
\cite[(10.1), p.\,845]{ACG},
 \cite[Prp.\,2.21 p.\,355]{CH},
\cite[Thm.\,2.8, p.\,8]{Ty},
 \cite[Thm.\,3.1, p.\,59]{Va}, and
\cite[Thm.\,1, p.\,215; Thm.\,2, p.\,220]{Z}.
  However, our particular approach and results appear to be new.

\begin{lemma} \label{prcodim}
  Assume $C\in |\lL|^\delta$.  Set
$\lambda:=\dim\Ker(\ecH^1(S,\,\oO_S)\to\ecH^1(S,\,\lL))$ and
$\alpha:=\dim\Ker(\ecH^1(C,\,\aA\cdot\oO_C(C))\to\ecH^1(C,\oO_C(C)))$.
 Then
\begin{gather}
 \delta-\alpha-\lambda\le\dim_C |\lL| - \dim_C|\lL|^\delta
         \le \delta \text{\quad and}
  \label{equpbd}\\
(\Cc_C|\lL|^\delta)_\red \subset \ecH^0(\wt C,\,\oO_{\wt C}(K_{\wt C}-n^*K)).
\label{eqconecnt}
\end{gather}

In addition, assume $\lambda=0$ and $\alpha=0$.   Then
\begin{equation}
(\Cc_C|\lL|^\delta)_\red = \ecH^0(\wt C,\,\oO_{\wt C}(K_{\wt C}-n^*K)).
\label{eqredSevcone}
\end{equation}

Finally, assume $C$ is immersed too.  Then $|\lL|^\delta$ is smooth at
$C$.
\end{lemma}

\begin{proof}
Plainly, $|\oO_S(C)|^\delta$ is, locally at $C$, the preimage of the
equigeneric locus $\Lambda^\delta$ in $\Lambda:=\Dfl(C)$.  As
codimension cannot increase on taking a preimage from a smooth ambient
target, the right-hand bound holds in \eqref{equpbd}.

In general, let $f\:X\to Y$ be a map of schemes, $x\in X$ a point,
$y:=f(x)\in Y$ the image.  Plainly, $f$ induces maps of tangent spaces
$\Tt_f\:\Tt_x(X)\to \Tt_y(Y)$ and tangent cones $\Cc_x(X)\to \Cc_y(Y)$,
so a map of reductions $\Cc_x(X)_\red\to \Cc_y(Y)_\red$.  Thus
$\Cc_x(X)_\red\subset \Tt_f^{-1}(\Cc_y(Y)_{\red})$.  Now, take
$|\lL|^\delta\to \Lambda^\delta$ for $f$, and take $C$ for $x$.
Therefore, $(\Cc_C|\lL|^\delta)_\red$ lies in the preimage of
$(\Cc_0\Lambda^\delta)_\red$ in $\Tt_C|\lL|^\delta$.  However,
$\Tt_C|\lL|^\delta\subset \Tt_C|\lL|$.  Thus $(\Cc_C|\lL|^\delta)_\red$
lies in the preimage of $(\Cc_0\Lambda^\delta)_\red$ in $\Tt_C|\lL|$.

Further, the tangent map $T_{C}|\lL| \to T_{C}\Lambda$ is given by this
composition:
\begin{equation}\label{eqtgtmap}
 \theta\:\ecH^0(S,\,\lL)\big/\Im \ecH^0(S,\oO_S) \overset\eta\Into
   \ecH^0(C,\,\oO_C(C)) \xto\nu \ecH^0(C,\,\oO_C / \jJ).
\end{equation}
 Therefore, \eqref{eqredtgtcone} and the injectivity of $\eta$ yield
\begin{equation}\label{eqredcones}
(\Cc_C|\lL|^\delta)_\red \subset \theta^{-1}\ecH^0(C,\,\aA/\jJ)
        \subset \nu^{-1}\ecH^0(C,\,\aA/\jJ).
\end{equation}

Consider  the following composition:
\begin{equation}\label{eqxinu}
\xi\:\ecH^0(C,\,\oO_C(C)) \xto\nu \ecH^0(C,\,\oO_C/\jJ)
 \xto\rho \ecH^0(C,\,\oO_C/\aA).
\end{equation}
The left-exactness of $\ecH^0$ yields $\ecH^0(C,\,\aA/\jJ)=\Ker\rho$ and
$\ecH^0(C,\,\aA\cdot\oO_C(C))=\Ker\xi$.  But $\nu^{-1}\Ker\rho=\Ker\xi$.
Hence $\nu^{-1}\ecH^0(C,\,\aA/\jJ)=\ecH^0(C,\,\aA\cdot\oO_C(C))$.  But
\eqref{eqZar} implies $\ecH^0(C,\,\aA\cdot\oO_C(C)) = \ecH^0(\wt
C,\,\oO_{\wt C}(K_{\wt C}-n^*K))$.  Thus \eqref{eqconecnt} holds.

The above considerations also yield $\nu^{-1}\ecH^0(C,\,\aA/\jJ)=\Ker\xi$.
So \eqref{eqredcones} yields
\begin{equation}\label{eqlbd1}
\dim_C|\lL|^\delta=\dim(\Cc_C|\lL|^\delta)_\red \le \dim\Ker\xi.
\end{equation}
On the other hand, the long exact cohomology sequences involving $\eta$
and $\xi$ yield
\begin{gather}
-\dim|\lL| + \dim \ecH^0(C,\,\oO_C(C))- \lambda =0 \label{gales1}\\
  \dim\Ker\xi -\dim \ecH^0(C,\,\oO_C(C)) + \dim\ecH^0(C,\,\oO_C/\aA)
    -\alpha = 0.\label{gales2}
\end{gather}
But $\dim\ecH^0(C,\,\oO_C/\aA)=\delta$.  Thus, combined, \eqref{eqlbd1}
and \eqref{gales1} and \eqref{gales2} yield the left-hand bound in
\eqref{equpbd}.

In addition, assume $\lambda=0$ and $\alpha=0$.   To prove
\eqref{eqredSevcone}, let's show  both sides of \eqref{eqconecnt}
are of the same dimension.  The left-hand side is of dimension
$\dim|\lL|-\delta$ by  \eqref{equpbd}.  On the other hand, \eqref{gales1}
and \eqref{gales2} yield $\dim\Ker\xi=\dim|\lL|-\delta$, and the
considerations after \eqref{eqxinu} show $\Ker\xi$ is equal to the
right-hand side, as desired.

Finally, assume $C$ is immersed too.  Then $\Lambda^\delta$ is smooth at
$C$ by Thm.\,2.59(1)(c) of \cite[p.\,355]{GLS}.  So $\Tt_0\Lambda^\delta
= H^0(C,\,\aA/\jJ)$ by \eqref{eqredtgtcone}.  Always, $\Tt_C|\lL|^\delta$ maps
into $\Tt_0\Lambda^\delta$; so $\Tt_C|\lL|^\delta$ lies in the preimage
$\Tt$ of $\Tt_0\Lambda^\delta$ in $\Tt_C|\lL|$.  But $\Tt$ is a vector
space of codimension $\delta$ owing to the above analysis; indeed,
$\Tt=\theta^{-1}\ecH^0(C,\,\aA/\jJ)$, and in \eqref{eqredcones}, the two
extreme terms are of codimension $\delta$.  But
$\codim\Tt_C|\lL|^\delta\le\delta$ by \eqref{equpbd}.  Thus
$\dim\Tt_C|\lL|^\delta=\dim_C|\lL|^\delta$.  Thus $|\lL|^\delta$ is smooth at
$C$.
\end{proof}

In the remaining two lemmas, we assume $S$ is {\it regular\/}; that
is, $\ecH^1(S,\,\oO_S) = 0$.  As a consequence,  in Thm.\,\ref{ththreshold}
and Prp.\,\ref{prNC}, instead of assuming $S$ is rational, we may assume
$S$ is regular.  But the ``generalization is illusory,'' as noted in
\cite[(v), p.\,116]{Tan} in a similar situation.  Indeed, assume
$\dim|\lL|\ge1$, else $\lL$ holds little interest.  Assume
$C\in|\lL|-V$.  Let $C'$ be its variable part, so that $|C'|$ has no
fixed components.  Then $C'$ is nonzero and nef.  Hence $\ecH^0(S,\,mK)
= 0$ for all $m\ge1$; else, $K\cdot C'\ge 0$, but $-K\cdot \Gamma\ge1$
for every component $\Gamma$ of $C'$ as $C\notin V$.  Since
$\ecH^1(S,\,\oO_S) = 0$, Castelnuovo's Criterion implies $S$ is
rational.

The first lemma below addresses the immersedness of a general member
of $|\lL|^\delta$.  The discussion involves another invariant of the
reduced curve $C$ on $S$, namely, the (total) multiplicity of its
Jacobian ideal $\jJ$, or what is the same, the colength of its extension
$\jJ\oO_{\wt C}$ to the normalization of $C$.  This invariant was
introduced by Teissier \cite[II.$6'$, p.\,139]{Te76} in order to
generalize Pl\"ucker's formula for the class (the degree of the dual) of
a plane curve.

This invariant was denoted $\kappa(C)$ by Diaz and Harris \cite[(3.2),
p.\,441]{DH}, but they defined it by the formula
$$\kappa(C)=2\delta(C)+m(C)$$ where $m(C)$ denotes the (total)
ramification degree of $\wt C/C$.  The two definitions are equivalent
owing to the following formula, due to Piene \cite[p.\,261]{P}:
\begin{equation}\label{eqRP}
\jJ\oO_{\wt C} = \aA\cdot \rR
\end{equation}
where  $\rR$ is the ramification ideal.

The invariant $\kappa(C)$ is upper semicontinuous in $C$; see
\cite[p.\,139]{Te76} or \cite[bot., p.\,450]{DH}.  So $|\lL|^\delta$
always contains a dense open subset $|\lL|^\delta_\kappa$ on which
$\kappa(C)$ is locally constant, termed an {\it equiclassical
locus\/} in \cite{DH}.

By definition, $C$ is immersed if and only if $m(C)=0$.  Thus if
$C\in|\lL|^\delta_\kappa$, then $\kappa(C)\ge 2\delta$, and $C$ is
immersed if and only if $\kappa(C)=2\delta$.  Further, if so, then every
curve $D$ in every component of $|\lL|^\delta_\kappa$ containing
$C$ is immersed.

\begin{lemma} \label{leimmersed} Assume $S$ regular, and
$C\in|\lL|^\delta_\kappa$.  Assume $-K\cdot C_1\ge1$ for every component
$C_1$ of $C$.  If some $C_1$ is not immersed, then $-K\cdot C_1=1$.
 \end{lemma}
 \begin{proof}
 Fix a $C_1$.  Assume $C_1$ is not immersed, but $-K\cdot C_1\ge2$.
Then there's a point $\wt P$ in the normalization of $C_1$ at which $n$
ramifies.  Set $\aA':=\aA\cdot n_*\oO_{\wt C}(-\wt P)$.  Then owing to
Lem.\,\ref{leVan}, the restriction map
$$\ecH^0(C,\,\oO_C(C))\to \ecH^0(C,\,\oO_C/\aA')$$
is surjective.  Since $S$ is regular, the following restriction map too
is surjective:
$$\ecH^0(S,\,\lL)\to\ecH^0(C,\,\oO_C(C)).$$

Set $\hH:=n_*(\jJ\oO_{\wt C})$.  Then $\aA'\supset \hH$ owing to Piene's
Formula \eqref{eqRP}.  But $\hH\supset \jJ$.  Set $\Lambda:=\Dfl(C)$.
It follows, as in the proof of Lem.\,\ref{prcodim}, that the image of
$T_C|\lL|$ in $T_0\Lambda$ is transverse to $\aA' / \jJ$.  Thus the
image of $|\lL|$ in $\Lambda$ contains a 1-parameter equigeneric family
whose tangent space at $0$ is transverse to $\aA'/\jJ$ inside $\aA/\jJ$.

Diaz and Harris \cite[(5.5), p.\,459]{DH} proved that $\hH/\jJ$ is the
reduced tangent cone to the locus of equiclassical deformations.  Thus
the above 1-parameter family exits $|\lL|^\delta_\kappa$ while remaining
in $|\lL|^\delta$, contrary to the openness of $|\lL|^\delta_\kappa$ in
$|\lL|^\delta$.
\end{proof}

Finally, we consider the smoothness over $\CC$ of the relative Hilbert
scheme of a family.  To be precise, given a family of curves with
parameter space $B$ and total space $\cC_B$, denote by $\cC^{[n]}_B$ the
relative Hilbert scheme of $n$ points.  Further, if $B\subset|\lL|$,
take $\cC_B$ to be the total space of the tautological family.

\begin{lemma} \label{prsmooth}
 Assume $S$ regular, and $-K\cdot C_1\ge1$ for every component $C_1$ of
$C$.  Fix $n\ge0$.  Then the relative Hilbert scheme
$\smash{\cC_{|\lL|}^{[n]}}$ is smooth over $\CC$ along the Hilbert
scheme $C^{[n]}$ of $C$ over $\CC$.
\end{lemma}
\begin{proof}
The proof has three steps: (1) show that $\cC_\Lambda^{[n]}$ is
 smooth over $\CC$ along $C^{[n]}$; (2) show that, for any point $z\in
C^{[n]}$, the image in $T_0\Lambda$ of the tangent space
$\smash{T_z\cC_\Lambda^{[n]}}$ contains $\aA/\jJ$ in $T_0\Lambda$; and
(3) show that $\smash{\cC_{|\lL|}^{[n]}}$ is smooth over $\CC$ along
$C^{[n]}$.  The hypothesis that $S$ is regular and $-K\cdot C_1\ge1$ is
not used in the first two steps.

Step (1) was done in \cite[Prp.\,17]{S}.  Here's the idea.  First, embed
$\cC_\Lambda^{[n]}$ in $S^{[n]}\times \Lambda$, where $S^{[n]}$ is the
Hilbert scheme.  The latter is smooth by Fogarty's theorem.  Form the
tangent bundle-normal bundle sequence (constructed barehandedly as (6)
in \cite{S}); it's the dual of the Second Exact Sequence of K\"ahler
differentials \cite[Prp.\,8.12, p.\,176]{H} .  It shows  the
question is local analytic about the singularities of $C$, as the
smoothness in question is equivalent to the surjectivity of the
right-hand map owing to \cite[(17.12.1)]{EGAIV4}.  So we may replace $C$
by an affine plane curve $\{f=0\}$.

Take a vector space $\VV$ of polynomials containing $f$ and also every
polynomial of degree at most $n$.  Form the tautological family
$\cC_\VV/\VV$.  Its relative Hilbert scheme $\smash{\cC_\VV^{[n]}}$ is
smooth over $\CC$ along $C^{[n]}$ owing to the analogous tangent
bundle-normal bundle sequence; its right-hand map is surjective by
choice of $\VV$.  Finally, as $\Lambda$ is versal, there's a map of
germs $\lambda\:(\VV,0)\to (\Lambda,0)$ such that $\cC_\VV^{[n]}$ is the
pullback of $\cC_\Lambda^{[n]}$.  It's smooth as the map on tangent
spaces is surjective.  Thus $\cC_\Lambda^{[n]}$ is smooth over $\CC$ along
$C^{[n]}$, as desired.

To do Step (2), we may assume that $z$ represents a subscheme $Z$ of $C$
supported on its singular locus $\Sigma$, because the map of tangent
spaces (essentially the map on the left in \cite[(6)]{S}\,)
is the product of the corresponding maps at the various points
$p$ in the support of $Z$, and these maps are clearly surjective at the
$p$ where $C$ is smooth.  Set $\oO:=\oO_{C,\Sigma}$, and let $I\subset
\oO$ be the ideal of $Z$.  Then $T_z\cC_\Lambda^{[n]}$ is the set of
first-order deformations of the inclusion map $I\into \oO$.  Further,
the map $T_z\cC_\Lambda^{[n]} \to T_0\Lambda$ forgets the inclusion, and
just keeps the deformation of $\oO$.

Let $J$ be the Jacobian ideal of $\oO$, the ideal of $\Sigma$.  Then
\eqref{eqtnsp} yields $T_0\Lambda=\oO/J$.  Further, let $A$ be the
conductor ideal of $\oO$.

The map $T_z\cC_\Lambda^{[n]} \to T_0\Lambda$ factors through the set
$D(\oO,I)$ of first-order deformations of the pair $(\oO,I)$ with $I$
viewed as an abstract $\oO$-module.  The map $D(\oO,I)\to T_0\Lambda$ was
studied by Fantechi, G\"ottsche and van Straten in \cite[Sec.\,C]{FGvS};
they showed that, in $\oO/J$,  the image of this map contains  $A/J$.

It remains to show $T_z\cC_\Lambda^{[n]} \to D(\oO,I)$ is surjective.
So take $(\oO',I')\in D(\oO,I)$.  As $\oO$ is Gorenstein,
$\Ext^1_\oO(I,\,\oO)=0$.  Hence, since deformations are flat, the
Property of Exchange \cite[Thm.\,(1.10)]{AK} implies this natural map is
bijective:
 \[ \rHom_\oO(I',\,\oO')\ox_{\oO'}\oO \risom \rHom_\oO(I,\,\oO).\] So
the inclusion map $I\into \oO$ lifts to a map $I'\to \oO'$.  The
latter is injective and its cokernel is flat owing to the Local
Criterion of Flatness, as $\oO'$ is flat and $I'\to \oO'$ reduces to
an injection with flat cokernel, namely, $I\into \oO$.

Finally, consider Step (3).  Since $\Lambda$ is versal, there exists a
map of germs $(|\lL|,C)\to (\Lambda,0)$ such that the germ
$(\cC_{|\lL|}^{[n]},z)$ is the pullback of the germ
$(\cC_{\Lambda}^{[n]},z)$, which is smooth over $\CC$ by Step~(1).
Since $(|\lL|,C)$ and $ (\Lambda,0)$ are smooth over $\CC$, the pullback
$(\cln,z)$ is therefore smooth over $\CC$ by general principles, if the
images in $T_0\Lambda$ of the tangent spaces $T_C|\lL|$ and $T_z
\cC^{[n]}_\Lambda$ sum to $T_0\Lambda$.

  Owing to \eqref{eqRes} and to Step (2), the latter holds if this
composition is surjective:
 \[ \ecH^0(S,\,\lL) \to \ecH^0(C,\,\oO_C(C)) \to \ecH^0(C,\,\oO_C / \aA).\]
 However, the first map is surjective as  $S$ is regular,
and the second map is surjective by Lem.\,\ref{leVan} with $\wt
M=\oO_{\wt C}$ owing to the hypothesis  $-K\cdot C_1\ge1$.
\end{proof}

\section{Proof of the main results}\label{Potmr}
Thm.\,\ref{ththreshold} can now be proved by revisiting the construction
in \cite{KST} of the universal polynomial $G_{\delta}(S,\lL)$ and making
use of the  lemmas in the preceding section.

\begin{proof}[{\scshape Proof of Thm.\,\ref{ththreshold}}] First,
\eqref{equpbd} yields $\codim_C|\lL|^\delta\le\delta$ for all
$C\in|\lL|^\delta$.  Also, $\ecH^1(S,\,\oO_S)=0$ as $S$ is rational,
and if $C\in(|\lL|^\delta-V)$, then $\ecH^1(C,\,\aA\cdot\oO_C(C))=0$
by \eqref{eqVan}; hence, if $C\in(|\lL|^\delta-V)$, then
\eqref{equpbd} yields $\codim_C|\lL|^\delta\ge\delta$.  Therefore, if
$\codim V>\delta$, then $\codim_C|\lL|^\delta=\delta$ for all $C$ in the
closure \/ $\bigl(|\lL|^\delta-V\bigr)\?{\phantom{I}}$, and then
$\bigl(|\lL|^\delta-V\bigr)\?{\phantom{I}} = |\lL|^\delta$.

Note that Lem.\,\ref{leimmersed} and the discussion before it imply
that, if $C \in \bigl(|\lL|^\delta-V\bigr)\?{\phantom{I}}$, then $C\in
|\lL|^\delta_\kappa$ if and only if $C$ is immersed, and that
$|\lL|^\delta_\kappa$ is open and dense in $|\lL|^\delta$.  Further, the
last assertion of Lem.\,\ref{prcodim} now implies $|\lL|^\delta_\kappa$
is smooth at $C$ if $C\notin V$.

It remains to compute $\deg |\lL|^\delta$ assuming $\codim V>\delta$.
Denote by $g$ the common arithmetic genus $p_aD$ of the $D\in|\lL|$.
Bertini's theorem \cite[Cor.\,10.9, p.\,274]{H} yields a $\delta$-plane
$\PP \subset |\lL|$ avoiding
$V\bigcup\big(|\lL|-|\lL|^\delta_\kappa\big)$ and such that
$\cC^{[n]}_{\PP}$ is  smooth over $\CC$ for $n\le g$.  But
$\cC^{[n]}_{\PP}$ is, by \cite[Thm.\,5, p.\,5]{AIK}, cut out of $\PP
\times S^{[n]}$, where $S^{[n]}$ is the Hilbert scheme, by a
transversally regular section of the rank-$n$ bundle $\lL^{[n]}$ that is
obtained by pulling $\lL$ back to the universal family and then pushing
it down.  Hence the topological Euler characteristic
$\chi(\cC^{[n]}_{\PP})$ can be computed by integrating polynomials in
the Chern classes of $\lL^{[n]}$ and $S^{[n]}$.  But, as Ellingsrud,
G\"ottsche, and Lehn \cite{EGL} show, such integrals admit universal
polynomial expressions in the Chern classes of $S$ and $\lL$.

Following \cite{KKV}, define $n_h(\PP)$ by this relation:
 \[\ts
\sum_{n=0}^\infty q^n\chi(\cC^{[n]}_{\PP})
 = \sum_{h=-\infty}^g n_h(\PP) q^{g-h} (1-q)^{2h-2}.\] For
$D\in|\lL|$, define $n_h(D)$ similarly.  By additivity of the
Euler characteristic, these definitions are compatible: $\chi(\PP,n_h) =
n_h(\PP)$ where $n_h\:\PP \to \ZZ$ is the constructible function $b
\mapsto n_h(\cC_b)$.  By \cite[App.\,B.1]{PT3},
  if $D$ is reduced of geometric genus $\wt g$, then $n_h(D) = 0$ for
$h<\wt{g}(D)$.  Thus the $n_h(\PP)$ admit universal polynomial expressions.

For each $\epsilon$, Lem.\,\ref{prcodim} implies $|\lL|^\epsilon$ is of
codimension $\epsilon$ at every $D\in (|\lL|^\epsilon-V)$.  So
$|\lL|^\epsilon-V$ is empty if $\epsilon>\dim|\lL|$.  Further, replacing
$\PP$ by a more general $\delta$-plane if neccessary, we may assume
$\PP\bigcap|\lL|^\epsilon$ is empty if $\delta<\epsilon\le\dim|\lL|$.
Then there are only finitely many $D\in\PP$ of cogenus $\delta$, and
none of greater cogenus.  Thus
 \[\ts
n_{g-\delta}(\PP) = \sum_{D \in \PP \cap
   |\lL|^\delta}  n_{\wt{g}}(D). \]

Alternatively, instead of using \eqref{equpbd} to bound the
$\codim|\lL|^\epsilon$, we could use \cite[Cor.\,9]{MS}, which asserts
that, given any family of locally planar curves whose $n$th relative
Hilbert scheme is smooth over $\cC$ and any $\epsilon\le n$, the curves
of cogenus $\epsilon$ form a locus of codimension at least  $\epsilon$
in the base.

Finally, as each $D \in \PP \cap |\lL|^\delta$ is immersed,
$n_{\wt{g}}(D) = 1$ by \cite[Eqn.\,5]{S} plus \cite[Prp.\,3.3]{B}.
Alternatively, this statement follows from \cite[Thm.\,A]{S}, because
$|\lL|^\delta$ is smooth at $D$.  Thus $n_{g-\delta}(\PP) =
\deg|\lL|^\delta$.  \end{proof} \vspace{-\medskipamount}

Lastly, we prove Prp.\,\ref{prNC}, which provides conditions under which
the nodal curves in the Severi variety $|\lL|^\delta$ form a dense open
subset $|\lL|^\delta_+$.  It is well known that $|\lL|^\delta_+$ is open and
dense if $S$ is the plane; see \cite[Thm.\,2, p.\,220]{Z} and
\cite[(10.7), p.\,847]{ACG} and \cite[Prp.\,2.2, p.\,355]{CH}.  Similar
arguments work if $S$ is a Hirzebruch surface; see \cite[Prp.\,8.1,
p.\,74]{Va}.  The broadest statement is given in \cite[Thm.\,2.8,
p.\,8]{Ty}.

However, even that statement is not broad enough to cover our needs.
Moreover, our approach appears to be new in places.  In addition, the
appendix develops the ideas in \cite{Ty} further, so as to provide
another proof of Prp.\,\ref{prNC} and the codimension statement in
Thm.\,\ref{ththreshold}.

\vspace{-\smallskipamount}
\begin{proof}[{\scshape Proof of Prp.\,\ref{prNC}}] Clearly,
$\deg|\lL|^\delta_+ = \deg |\lL|^\delta$ if $|\lL|^\delta_+$ is open and
dense in $|\lL|$.  Thus Thm.\,\ref{ththreshold} and the first assertion of
Prp.\,\ref{prNC} yield the second.

To prove the first assertion, assume $C\in|\lL|^\delta$, fix $P\in C$, and
consider the local Milnor number $\mu(C,P)$, which vanishes if  $C$ is
smooth at $P$.  It is, by \cite[Thm.\,2.6(2), p.\,114]{GLS}, upper
semicontinuous in this sense: there is an (analytic or \'etale)
neighborhood $B$ of the point in $|\lL|^\delta$ representing $C$ and a
neighborhood $U$ of $P$ in the tautological total space $\cC_B$ such
that, for each $b\in B$,
\abovedisplayskip=3.5pt plus 3pt\belowdisplayskip=\abovedisplayskip
 \begin{equation}\label{eqmuuc} \ts
\mu(C,P)\ge\sum_{Q\in\cC_b\cap U}\mu(\cC_b,Q).
\end{equation}
So the total Milnor number $\mu(C):=\sum_z\mu(C,z)$ too is upper
semicontinuous in $C$.

Therefore, $|\lL|^\delta$ always contains a dense open subset
$|\lL|^\delta_\mu$ on which $\mu(C)$ is locally constant.  So fix
$C\in|\lL|^\delta_\mu$.  Then after $B$ is shrunk, equality holds in
\eqref{eqmuuc}.  Therefore, there is a section $B\to \cC_B$ along which
the family is equisingular by work of Zariski's \cite{Z2, Z70}, of L\^e
and Ramanujam's \cite{LR} and of Teissier's\,---\,see both
\cite[Prp.\,2.62, p.\,359]{GLS} and \cite[Thm.\,5.3.1, p.\,123]{Te76},
as well as the historical note \cite[5.3.10, p.\,129]{Te76}.

Consider Milnor's Formula $\mu(C) =2\delta-\sum_{Q\in C}(r(Q)-1)$ where
$r(Q)$ is the number of branches of $C$ at $Q$; see \cite[Prp.\,3.35,
p.\,208]{GLS}.  It implies $\mu(C)\ge\delta$, with equality if and only
if $C$ is $\delta$-nodal.  So the nodal locus $|\lL|^\delta_+$ is always
a union of components of $|\lL|^\delta_\mu$.  Thus to complete the proof
of Prp.\,\ref{prNC}, it suffices to show $|\lL|^\delta_\mu-V$ consists
of nodal curves.  So assume $C\in |\lL|^\delta_\mu-V$.

First of all, $C$ is immersed by Lem.\,\ref{leimmersed}.  So
Lem.\,\ref{prcodim} implies $|\lL|^\delta_\kappa$ is smooth at $C$ with
tangent space equal to $\ecH^0(\wt C,\,\oO_{\wt C}(K_{\wt C}-n^*K))$.

Form the composition $B\to \cC_B\to S$ of the above equisingular section
and of the projection.  Denote the preimage of $P\in S$ by $B'$.
Evidently, $\dim_C B - \dim_C B' \le 2$.

By equisingularity, $P$ has the same multiplicity $m$ on every
$D\in B'$.  Denote by $S'$ the blowup of $S$ at $P$, by $E$ the exceptional
divisor, by $C'$ the strict transform of $C$.  Set $\lL':=\oO_{S'}(C')$
and $\delta':= \delta-m(m-1)/2$.  Taking strict transforms gives a map
$B'\to |\lL'|^{\delta'}$.  It is injective as taking images gives an
inverse.

Denote by $n'\:\wt C\to C'$ the normalization map, by $K'$ the
canonical class of $S'$.  Then \eqref{eqconecnt} yields
 $(\Cc_{C'}|\lL'|^{\delta'})_\red \subset
 \ecH^0(\wt C,\,\oO_{\wt C}(K_{\wt C}-n^{\prime*}K'))$.  Therefore,
\begin{align}
\dim\ecH^0(\wt C,\,&\oO_{\wt C}(K_{\wt C}-n^*K))\label{eqH02}
 - \dim\ecH^0(\wt C,\,\oO_{\wt C}(K_{\wt C}-n^{\prime*}K'))\\
   &{\le \dim_C|\lL|^\delta-\dim_{C'} |\lL'|^{\delta'}
        = \dim_C B - \dim_C B' \le 2.}\label{eqby2}
\end{align}
\goodbreak

The groups in \eqref{eqH02} belong to the long exact cohomology sequence
arising from
$$0\to\oO_{\wt C}(K_{\wt C}-n^{\prime*}K')\to\oO_{\wt C}(K_{\wt C}-n^*K)
   \to\oO_{n^{\prime*}E}\to 0.$$
Further, $\ecH^1(\wt C,\,\oO_{\wt C}(K_{\wt C}-n^*K))=0$ by
\eqref{eqVan} as $C\notin V$.  Hence \eqref{eqH02} is equal to
\begin{equation}\label{eq2eqto}
\dim \ecH^0(\oO_{n^{\prime*}E})
 - \dim \ecH^1(\wt C,\,\oO_{\wt C}(K_{\wt C}-n^{\prime*}K')).
\end{equation}
But $\deg(n^{\prime*}E)=m$; so $\dim \ecH^0(\oO_{n^{\prime*}E})=m$.

Denote by $ C_1,\dotsc, C_h$ the components of $ C$, by $\wt C_i$
the normalization of $C_i$.  Set
$$k_i:=-K\cdot C_i \text{\quad and\quad} m_i:=\mult(P,\,C_i) =
  \deg(n^{\prime*}E|\wt C_i)\ge0.$$
Now, $n^{\prime*}K'=n^*K+n^{\prime*}E$.  Therefore, \eqref{eqZar} and
\eqref{eqVan} yield
$$ \dim \ecH^1(\wt C,\,\oO_{\wt C}(K_{\wt C}-n^{\prime*}K'))
\ts   \le  \sum_{i=1}^h\max(0,\,1-k_i+m_i).$$
Note $m=\sum_{i=1}^h m_i$.  Consequently, \eqref{eq2eqto} and
\eqref{eqby2} yield
\begin{equation}\label{eqbd3}
\ts \sum_{i=1}^h s_i\le2
 \text{\quad where\quad} s_i:=m_i-\max(0,\,1-k_i+m_i).
\end{equation}

Note $m_i\ge s_i\ge0$ for all $i$, as $0\le\max(0,\,1-k_i+m_i)\le m_i$
since $k_i\ge1$ owing to (2) of Thm.\,\ref{ththreshold}.  Also, $s_i=0$
if $k_i=1$ for any $i$ and any $m_i$; conversely, if $s_i=0$ and
$m_i\ge1$, then $k_i=1$.  Further, $m_i=s_i$ if and only if $k_i\ge
m_i+1$, as both conditions are obviously equivalent to
$\max(0,\,1-k_i+m_i)=0$.  Clearly, $k_i\le m_i+1$ if and only if
$s_i=k_i-1$.

Using \eqref{eqby2}, let's now rule out $m\ge3$.  Aiming for a
contradiction, assume
\begin{equation}\label{eqms}
m_1\ge\dotsb \ge m_h \tq{and}m_1+\dotsb+m_h=m\ge3.
\end{equation}
Now, \eqref{eqbd3}
yields $s_1\le 2$ as $s_i\ge0$ for all $i$.  So if $m_1\ge3$, then
$m_1-2\le1-k_1+m_1$; whence, $k_1\le3$, contrary to (4) of
Prp.\,\ref{prNC}.  Thus \eqref{eqms} yields $2\ge m_1\ge m_2\ge1$.

Suppose $m_1=2$.  Then (5) of Prp.\,\ref{prNC} rules out $k_1=1$ and
$k_2=1$.  So $k_1\ge2$ and $k_2\ge 2$.  Suppose $k_1=2$.  Then $s_1=1$.
So $s_2\le1$.  If $m_2=2$, then $k_2=2$, contrary to (6) of
Prp.\,\ref{prNC}.  If $m_2=1$, then already $k_1=2$ is contrary to (7)
of Prp.\,\ref{prNC}.  Suppose $k_1\ge3$.  Then $s_1=2$.  So $s_2=0$.  So
$k_2=1$.  But this case was already ruled out.  Thus the case $m_1=2$ is
ruled out completely.

Lastly, suppose $m_1=1$.  Then \eqref{eqms} yields $m_2=1$ and $m_3=1$
too.  So (8) of Prp.\,\ref{prNC} yields $k_i\ge2$ for $i=1,2,3$.  Hence
$s_i=1$ for $i=1,2,3$, contradicting \eqref{eqbd3}.  Thus $m=2$, as
claimed.

Finally, given $m=2$, let's show $P$ is a simple node.  Since $C$ is
immersed at $P$, it is locally analytically given by an equation of the
form $y^2 = x^{2k}$ for some $k\ge1$.  Denote by $\wt P,\,\wt Q\in\wt C$
the points above $P$ on the branches with equations $y=x^k$ and $y=-x^k$.
Then \eqref{eqZar} and \eqref{eqVan} imply
$$ \dim \ecH^1(\wt C,\,\oO_{\wt C}(K_{\wt C}-n^*K-\wt P-\wt Q))=0,$$
because $k_i\ge1$ for all $i$ owing to (2) of Thm.\,\ref{ththreshold}
and because either $k_i\ge2$ for $i=1,2$ if $\wt P\in \wt C_1$ and $\wt
Q\in \wt C_2$ owing to (9) of Prp.\,\ref{prNC} or $k_1\ge3$ if $\wt
P,\wt Q\in \wt C_1$ owing to (10) of Prp.\,\ref{prNC}.  Hence the
following restriction map is surjective:
$$\ecH^0(\wt C,\,\oO_{\wt C}(K_{\wt C}-n^*K))
 \onto \ecH^0(\wt C,\,\oO_{\wt P+\wt Q}).$$

Therefore, there's a section of $\oO_{\wt C}(K_{\wt C}-n^*K)$ that
doesn't vanish at $\wt P$, but does at $\wt Q$.  Correspondingly,
there's a first-order deformation of $C$.  Say it's given locally by
$y^2 - x^{2k} + \epsilon g(x,y)$.  Then $g(t,t^k)$ is of degree $k$, but
$g(t,-t^k)$ is of degree $k+1$.  Clearly, any such $g$ is, up to scalar
multiple, necessarily of the form $g(x,y) = x^k + y + O(x^{k+1}, y^2)$.
However, the Jacobian ideal of the singularity is $\langle y,\,
x^{2k-1}\rangle$.  This ideal must contain $g(x,y)$ as the deformation under
consideration is equisingular and as the Jacobian ideal is equal to the
equisingular ideal by \cite[Lem.\,2.16, p.\,287]{GLS}.  Hence $k=1$.
Thus $P$ is an simple node of $C$, as desired.
\end{proof}

\appendix
\section{An alternative proof by Ilya Tyomkin}\label{apap}

Our goal is to use the deformation theory of maps to provide an
alternative proof of Prp.\,\ref{prNC} and the codimension statement in
Thm.\,\ref{ththreshold}. The general idea goes back to Arbarello and
Cornalba \cite{AC}, but the proof contains new ingredients, most of
which were introduced in \cite{Ty}.

\subsection{Notation} \noindent\medskip

Let $\delta,\, S,\, \lL,\, K,\, |\lL|^\delta,\, |\lL|^\delta_+$ be as in
the Introduction.  Again, we work over complex numbers $\CC$, but as is
standard, we denote the residue field at a point $p$ by $k(p)$.
Moreover, as our treatment is purely algebraic, all the statements and
proofs are valid over an arbitrary algebraically closed field of
characteristic 0.

Given a morphism $f\: X\to Y$, and $p^1,\dotsc,p^r\in X$ points where
$X$ is smooth, $\Def(X,f;\,\underline{p})$ denotes the functor of
deformations of $(X,f;\,p^1,\dotsc, p^r)$; i.e., if $(T,0)$ is a local
Artinian $\CC$-scheme, then $\Def(X,f;\,\underline{p})(T,0)$ is the set
of isomorphism classes of this data: $(X_T,f_T;p^1_T,\dotsc,
p^r_T;\iota)$ where $X_T$ is $T$-flat, each $p^i_T\:T\to X_T$ is a
section, $f_T\:X_T\to Y\times T$ is a $T$-morphism, and $\iota$ is
an isomorphism
 $$\iota\:(X_0,f_0;p^1_0,\dotsc,p^r_0)\risom (X,f;\,p^1,\dotsc,p^r).$$
Let $\Def^1(X,f;\,\underline{p})$ denote the set of first-order
deformations $\Def(X,f;\,\underline{p})(T,0)$ where $T:=\Spec(\CC[\epsilon])$
and  $\CC[\epsilon]$ is the ring of dual numbers.

If $X$ and $Y$ are smooth, set $\nN_f:=\Coker(T_X\to f^*T_Y)$; it's
the \textit{normal sheaf}.

\subsection{Three Lemmas}
\begin{lemma}\label{cl:def1}
Let $(C;p^1,\dotsc, p^r)$ be a smooth curve with marked points, and $f\:C\to
S$ a map that does not contract components of $C$. Then
there is a natural exact sequence
 $$\ts 0\to \bigoplus_{i=1}^r
\Tt_{p^i}(C)\to {\Def}^1(C,f;\,\underline{p})\to H^0(C, \nN_f)\to 0.$$
\end{lemma}
\begin{proof}
Consider the forgetful map $\phi\:{\Def}^1(C,f;\,\underline{p})\to
{\Def}^1(C,f)$.  It is surjective by the infinitesimal lifting property,
since $C$ is smooth at all the $p^i$. Its kernel
is canonically isomorphic to $\bigoplus_{i=1}^r{\Def}^1(p^i\to
C)$, so to $\bigoplus_{i=1}^r \Tt_{p^i}C$. Finally, since $T_C\to f^*T_S$ is
injective, ${\Def}^1(C,f)=\Ext^1(\Ll_{C/S},\oO_C)=H^0(C, \nN_f)$, where
$\Ll_{C/S}$ is the cotangent complex of $f\:C\to S$; see
\cite[(2.1.5.6), p.\,138; Prp.\,3.1.2, p.\,203; Thm.\,2.1.7, p.\,192]{I71-72}
or \cite[pp.\,374--376]{H73}.
\end{proof}

\begin{lemma}\label{cl:resdef}
Let $C$ be a smooth curve, $f\:C\to S$ a map, $D\subset S$ a closed
curve.  Set $Z:=D\times_SC$, and assume $Z$ is reduced and
zero-dimensional.  Let $g\:Z\to D$ be the inclusion, and set
$T:=\Spec(\CC[\epsilon])$ and $(Z_T,g_T):=(C_T,f_T)\times_{S\times
T}(D\times T)$.  Then sending $(C_T,f_T)$ to $(Z_T,g_T)$ defines a map
$d\psi\:\Def^1(C,f)\to \Def^1(Z,g)$. Furthermore,
$d\psi(H^0(C,\nN_f^{\rm tor}))=0$.
\end{lemma}
\begin{proof}
To prove $d\psi$ is well defined, it suffices to show that $Z_T$ is
$T$-flat. Let $0\in T$ be the closed point, $q\in Z\subset Z_T$ a
preimage of $0$, and $h=0$ a local equation of $D$ at $f(q)$. Then there
exists an exact sequence $0\to \oO_{C_T,q}\to \oO_{C_T,q}\to
\oO_{Z_T,q}\to 0$ where the first map $m_h$ is the multiplication by
$f_T^*(h)$.  Also, $m_h\otimes k(0)\:\oO_{C,q}\to \oO_{C,q}$ is
injective, since the locus of zeroes of $f^*(h)$ in $C$ is of
codimension 1, and so $f^*(h)\in \oO_{C,q}$ is not a zero-divisor. Thus,
$\oO_{Z_T,q}$ is flat by the local criterion of flatness
\cite[Cor. 5.7]{SGA1}.  Thus $d\psi$ is well defined.

As $Z$ is reduced, $Z\bigcap\Supp(\nN_f^{\rm tor}) =\emptyset$. Set
$U:=C\setminus \Supp(\nN_f^{\rm tor})$. Then $d\psi$ factors through $\Def^1(U,
f|_U)=\nN_f(U)=\big(\nN_f/\nN_f^{\rm tor}\big)(U)$.  Thus
$d\psi(H^0(C,\nN_f^{\rm tor}))=0$.
\end{proof}

\begin{lemma}\label{cl:etneigh}
Let $W$ be an algebraic variety, $C_W\to W$ a flat family of reduced
curves, $\widetilde{C}_W\to C_W$ the normalization, and $Z_W\subset
\widetilde{C}_W$ a reduced closed subvariety quasi-finite over
$W$. Then there exists an \'etale morphism $U\to W$ and sections $s_i\:
U\to \widetilde{C}_U$ such that the following two conditions hold: \(1)
$C_U\to U$ is equinormalizable, i.e., $\widetilde{C}_U\to U$ is flat and
$\widetilde{C}_u\to C_u$ is the normalization for any $u\in U$; and \(2)
$Z_U\to U$ is \'etale and $Z_U=\cup_{i=1}^rs_i(U)$.
\end{lemma}

\begin{proof}
The generic fiber $\widetilde{C}_\eta$ is normal since normalization
commutes with arbitrary localizations. Then it is geometrically normal,
since the characteristic is zero; and hence $\widetilde{C}_\eta\to\eta$
is smooth by flat descent. Then $\widetilde{C}_W\to W$ is generically
smooth by generic flatness theorem, i.e., there exists an open dense
subset $U_0\subset W$ such that $\widetilde{C}_{U_0}\to U_0$ is
smooth. In particular, $\widetilde{C}_{U_0}\to U_0$ is flat and has
normal fibers. But, $\widetilde{C}_u\to C_u$ is finite for any $u\in
U_0$, and hence the normalization. Furthermore, for any \'etale map
$U\to U_0$, the family $C_U\to U$ is equinormalizable since
normalization commutes with \'etale base changes.

The morphism $Z_W\to W$ is finite, and $Z_W$ is reduced. Thus,
$Z_\eta\to \eta$ is finite and \'etale since the characteristic is
zero. Hence, after shrinking $U_0$, we may assume that $Z_{U_0}\to U_0$
is finite and \'etale. Then there exists an \'etale morphism $U\to U_0$
such that $Z_U$ is the disjoint union of $\deg(Z_\eta\to\eta)$ copies of
$U$ and the map $Z_U\to U$ is the natural projection. Hence $U$ is as
needed.
\end{proof}
\subsection{The results}
\begin{proposition}\label{prop:dimbd}
Let $W\subseteq |\lL|^\delta$ be an irreducible
subvariety, $C_W\to W$ the tautological family of curves,
$\widetilde{C}_W\to C_W$ the normalization, $f_W\:\widetilde{C}_W\to
S$ the natural morphism, and $0\in W$ a {\em general} closed
point.  Assume that $C_0$ is reduced.

\(1) Then there  exists a natural embedding $\Tt_0(W)\into
H^0(\widetilde{C}_0,\nN_{f_0}/\nN_{f_0}^{\rm tor})$.

\(2) If $-K.C\ge 1$ for any irreducible component $C\subseteq C_0$, then
\begin{equation}\label{eq:dimineq} \dim(W)\le
h^0(\widetilde{C}_0,\nN_{f_0}/\nN_{f_0}^{\rm tor})\le
-K.C_0+p_g(C_0)-1.
\end{equation}

\(3) If \eqref{eq:dimineq} is equality and $-K.C>1$ for an irreducible
component $C$ of $C_0$, then $C$ is immersed.

\(4) If \eqref{eq:dimineq} is equality and $-K.C>1$ for any irreducible
component $C$ of $C_0$, then $\nN_{f_0}$ is invertible and $\Tt_0(W)\to
H^0(\widetilde{C}_0,\nN_{f_0})$ is an isomorphism.
\end{proposition}

\begin{proof}
Pick a smooth irreducible closed curve $D\subset S$ in a very ample
linear system such that $h^0(S,\lL(-D))=0$. Then $D\cap C_w$ is finite
for any $w\in W$, and is reduced for almost all $w\in W$ by Bertini's
theorem. In particular, $D\cap C_0$ is reduced since $0\in W$ is
general. Hence the projection $Z_W:=\widetilde{C}_W\times_S D\to W$ is
finite, since it is a projective morphism with finite fibers. Let
$g_0\:Z_0\to D$ be the closed immersion. Then, by Lem.\,\ref{cl:resdef}
and Lem.\,\ref{cl:etneigh}, there exists a commutative diagram
\begin{equation}\UseTips
\xymatrix{
\Tt_0(W)\ar@{^(->}[d]\ar[r] & \Def^1(\widetilde{C}_0, f_0)\ar[d]\ar[r]
 & H^0(\widetilde{C}_0,\nN_{f_0}/\nN_{f_0}^{\rm tor})\ar[d] \\
\Tt_{Z_0}(|\lL\otimes \oO_D|)\ar@{^(->}[r] & \Def^1(Z_0,g_0)\ar@{=}[r] &
H^0(Z_0,\nN_{g_0})
}
\end{equation}
where $\Tt_0(W)\to \Tt_{Z_0}(|\lL\otimes \oO_D|)$ is injective since
$W\subseteq |\lL|\subseteq |\lL\otimes \oO_D|$ by the choice of $D$; and
$\Tt_{Z_0}(|\lL\otimes \oO_D|)\to \Def^1(Z_0,g_0)$ is injective since
$\Tt_{Z_0}(|\lL\otimes \oO_D|)$ is a subspace of the space of
first-order embedded deformations of $g_0(Z_0)\subset D$, and the latter
is canonically isomorphic to $\bigoplus_{p\in
g_0(Z_0)}\Tt_p(D)=\Def^1(Z_0,g_0)$. Thus, the composition $\Tt_0(W)\to
H^0(Z_0,\nN_{g_0})$ is injective, and hence so is $\Tt_0(W)\to
H^0(\widetilde{C}_0,\nN_{f_0}/\nN_{f_0}^{\rm tor})$ as asserted by (1).

(2) The first inequality in \eqref{eq:dimineq} follows from (1). Since
both sides of the second inequality in \eqref{eq:dimineq} are additive
with respect to unions, we may assume that $C_0$ is irreducible. Let
$0\to \nN_{f_0}/\nN_{f_0}^{\rm tor}\to\fF$ be an invertible extension
such that $c_1(\fF)=c_1(\nN_{f_0})$. By the assumption,
$c_1(\fF)=c_1(\nN_{f_0})=c_1(\omega_{\widetilde{C}_0})-K.C_0>
c_1(\omega_{\widetilde{C}_0})$. Thus,
$h^0(\widetilde{C}_0, \nN_{f_0}/\nN_{f_0}^{\rm tor})\le
h^0(\widetilde{C}_0, \fF)=c_1(\fF)+1-p_g(C_0)=-K.C_0+p_g(C_0)-1$ by
Riemann--Roch theorem, since
$h^0(\widetilde{C}_0,\fF^\vee\otimes\omega_{\widetilde{C}_0})=0$; and
hence \eqref{eq:dimineq} holds.

(3) Once again, we may assume that $C_0$ is irreducible. To prove that
$C_0$ is immersed, it is sufficient to show that $\nN_{f_0}^{\rm
tor}=0$. Assume to the contrary that $\nN_{f_0}^{\rm tor}\ne 0$. Pick an
invertible extension $0\to \nN_{f_0}/\nN_{f_0}^{\rm tor}\to\fF$ with
$c_1(\fF)=c_1(\nN_{f_0})-1$. By the assumption,
$c_1(\fF)=c_1(\nN_{f_0})-1=c_1(\omega_{\widetilde{C}_0})-K.C_0-1>
 c_1(\omega_{\widetilde{C}_0})$. Thus,
$h^0(\widetilde{C}_0, \nN_{f_0}/\nN_{f_0}^{\rm tor})\le
h^0(\widetilde{C}_0, \fF)=c_1(\fF)+1-p_g(C_0)=-K.C_0+p_g(C_0)-2$ by
Riemann--Roch theorem, which is a contradiction.

(4) Note that by (3) we have: $\nN_{f_0}^{\rm tor}=0$, and hence
$\nN_{f_0}$ is invertible. Then by (2),
$\dim(\Tt_0(W))=h^0(\widetilde{C}_0, \nN_{f_0}/\nN_{f_0}^{\rm
tor})=h^0(\widetilde{C}_0, \nN_{f_0})$. Thus, (1) implies that
$\Tt_0(W)\into
H^0(\widetilde{C}_0,\nN_{f_0})=H^0(\widetilde{C}_0,\nN_{f_0}/\nN_{f_0}^{\rm
tor})$ is an isomorphism.
\end{proof}

\begin{remark}\label{rem:dimcodim}
By definition,  $\delta:=p_a(C_0)-p_g(C_0)$.  Hence,
if $S$ is rational and $-K.C_0\ge 1$, then the adjunction formula and
Lem.\,\ref{lecdB} yield
\begin{equation*}
 -K.C_0+p_g(C_0)-1=C_0.(C_0-K)/2-\delta=\dim|\lL|-\delta.
\end{equation*}
\end{remark}

\begin{proposition}\label{prop:nodality} Fix a point $q\in S$, and a curve
$E\subset S$. Let $W\subseteq |\lL|^\delta$ be an irreducible
subvariety, $C_W\to W$ the tautological family of curves,
$\widetilde{C}_W\to C_W$ the normalization, $f_W\:\widetilde{C}_W\to
S$ the natural morphism, and $0\in W$ a {\em general} closed
point. Assume that $C_0$ is reduced and immersed,
$\dim(W)=-K.C_0+p_g(C_0)-1$, and $-K.C\ge 1$ for any irreducible
component $C$ of $C_0$.

\(1) If $-K.C>1$ for any irreducible component $C$ of $C_0$, then $q\notin
C_0$.

\(2) Let $q_0\in C_0$ be a point of multiplicity at least three, and
$p^1,p^2,p^3\in \widetilde{C}_0$ three distinct preimages of
$q_0$. Then there exists an irreducible component $C\subseteq C_0$ such
that $-K.C\le\left|\widetilde{C}\cap\{p^1,p^2,p^3\}\right|$.

\(3) Let $q_0\in C_0$ be a singular point with at least two tangent
branches, and $p^1,p^2\in \widetilde{C}_0$ the preimages of $q_0$ on
these branches. Then there exists an irreducible component $C\subseteq
C_0$ such that $-K.C\le\left|\widetilde{C}\cap\{p^1,p^2\}\right|$.

\(4) If $-K.C>1$ for any irreducible component $C$ of $C_0$, then any
branch of $C_0$ intersects $E$ transversally. Furthermore, if $C_0^{\rm
sing}\cap E\ne\emptyset$, then there exists an irreducible component
$C\subseteq C_0$ such that $C^{\rm sing}\cap E\ne\emptyset$ and
$-K.C=2$.
\end{proposition}

\begin{proof}
First, note that $\nN_{f_0}$ is invertible since $C_0$ is immersed. Thus,
the embedding $\Tt_0(W)\into
H^0(\widetilde{C}_0,\nN_{f_0}/\nN_{f_0}^{\rm tor})=H^0(\widetilde{C}_0,
\nN_{f_0})$ of Prp.\,\ref{prop:dimbd} (1) is an isomorphism by
Prp.\,\ref{prop:dimbd} (2). Then $h^0(\widetilde{C}_0, \nN_{f_0})=-K.C_0+p_g(C_0)-1=\chi(\nN_{f_0})$, and hence $h^1(\widetilde{C}_0,
\nN_{f_0})=0$. Let $A_W\subset C_W$ be the locus of
singular points of the fibers $C_W\to W$. Set
$Z_W:=\nu^{-1}(A_W)\cup f_W^{-1}(q\cup E)\subset \widetilde{C}_W$, where
$\nu\:\widetilde{C}_W\to C_W$ is the normalization. Then $Z_W\subset
\widetilde{C}_W$ is locally closed, and $Z_W\to W$ has finite
fibers. Thus, by Lem.\,\ref{cl:etneigh}, there exists an \'etale
neighborhood $U$ of $0$ and disjoint sections $s_i\:U\to
\widetilde{C}_U$ such that $Z_U=\cup_{i=1}^r s_i(U)$. Set
$p^i:=s_i(0)$. Then the isomorphism $\Tt_0(W)\to H^0(\widetilde{C}_0,
\nN_{f_0})=\Def^1(\widetilde{C}_0,f_0)$ factors through
$\Def^1(\widetilde{C}_0,f_0; \underline{p})$ for any $1\le i_1<\cdots
<i_m\le r$, where $\underline{p}=(p^{i_1},\dotsc,p^{i_m})$.

Consider the exact sequence of Lem.\,\ref{cl:def1}
$$0\to \oplus_{j=1}^m\left(T_{\widetilde{C}_0}\otimes k(p^{i_j})\right)
\to {\Def}^1(\widetilde{C}_0,f_0; \underline{p})\to
 H^0(\widetilde{C}_0, \nN_{f_0})\to 0,$$
the restriction map $\gamma\: H^0(\widetilde{C}_0,
\nN_{f_0})\to\oplus_{j=1}^m\left(\nN_{f_0}\otimes k(p^{i_j})\right)$,
and the forgetful map $$\beta\:{\Def}^1(\widetilde{C}_0,f_0;
\underline{p})\to
\oplus_{j=1}^m{\Def}^1(p^{i_j},f_0|_{p^{i_j}})=\oplus_{j=1}^m\left(f_0^*T_S\otimes
k(p^{i_j})\right).$$
Then the following diagram is commutative:
\begin{equation}\UseTips\label{eq:diagdif}
\xymatrix{
\Tt_0(W)\ar[r]\ar@{=}[d] &\Def^1(\widetilde{C}_0,f_0;
 \underline{p})\ar[d]\ar[r]^-\beta & \bigoplus_{j=1}^m\left(f_0^*T_S
\otimes k(p^{i_j})\right)\ar[d]^\pi\\
\Tt_0(W)\ar[r]^-\sim & H^0(\widetilde{C}_0, \nN_{f_0})\ar[r]^-\gamma &
\bigoplus_{j=1}^m\left(\nN_{f_0}\otimes k(p^{i_j})\right)
}
\end{equation}

From the long exact sequence of cohomology associated to the short exact
sequence of sheaves $0\to \nN_{f_0}(-\sum_{j=1}^m p^{i_j})\to
\nN_{f_0}\to\oplus_{j=1}^m\nN_{f_0}\otimes k(p^{i_j}) \to 0$ we obtain:
 $\Ker(\gamma)=H^0(\widetilde{C}_0, \nN_{f_0}(-\sum_{j=1}^m p^{i_j}))$,
and $\Coker(\gamma)\subseteq H^1(\widetilde{C}_0,
\nN_{f_0}(-\sum_{j=1}^m p^{i_j}))$. Thus, the map $\gamma$ is not
surjective if and only if $h^1(\widetilde{C}_0, \nN_{f_0}(-\sum_{j=1}^m
p^{i_j}))\ne 0$, since $h^1(\widetilde{C}_0,
\nN_{f_0})=0$. In particular, if $\gamma$ is not surjective then there
exists an irreducible component $C\subseteq C_0$ such that
$c_1(\nN_{f_0}(-\sum_{j=1}^m p^{i_j})|_{\widetilde{C}})\le
c_1(\omega_{\widetilde{C}})$, or, equivalently, $-K.C\le
|\widetilde{C}\cap \{p^{i_j}\}_{j=1}^m|$.

Let $q_0\in C_0$ be either a singular point, or a point of intersection
$C_0\cap E$, or $q_0=q$. Assume that $\nu(s_i(0))=q_0$ for $1\le i\le
m$. Since $0\in W$ is general, $\nu\circ s_i=\nu\circ s_j$ for all $1\le
i\le j\le m$. Set $i_j:=j$, and consider diagram~\eqref{eq:diagdif}.
For any $1\le j\le m$, the tensor product $f_0^*T_S\otimes k(p^j)$ is
canonically isomorphic to $\Tt_{q_0}(S)=\Def^1(q_0\to S)$, and $\beta$
factors through the diagonal map $\Delta\:\Tt_{q_0}(S)\to
\bigoplus_{j=1}^m\left(f_0^*T_S\otimes k(p^j)\right)$. Hence
$\Im(\gamma)\subseteq \Im(\pi\circ\Delta)$.

(1) Assume to the contrary that $q\in C_0$, and set $q_0:=q$. Without loss
of generality, $\nu(s_1(0))=q_0$. Set $m:=1$, and
consider diagram \eqref{eq:diagdif}. Then the image of $\Tt_0(W)$ in $\Def^1(p^1,f_0|_{p^1})$ is trivial since $q$ is fixed. Thus,
$\gamma$ is the zero map, and hence there exists an irreducible component $C\subseteq C_0$ such that $-K.C\le 1$, which is a contradiction.

(2) Assume that $q_0\in C_0$ is a singular point of multiplicity at
least three. Without loss of generality, $s_1(0), s_2(0), s_3(0)$ are
preimages of $q_0$. Set $m:=3$, and consider diagram
\eqref{eq:diagdif}. Then $\dim(\Im(\gamma))\le
\dim(\Im(\pi\circ\Delta))=2<3$, and hence $\gamma$ is not
surjective. Thus, there exists an irreducible component as asserted.

(3) Assume that $C_0$ has at least two tangent branches at
$q_0$. Without loss of generality, $s_1(0), s_2(0)$ are the preimages of
$q_0$ on the tangent branches. Set $m:=2$, and consider diagram
\eqref{eq:diagdif}. Then $\dim(\Im(\gamma))\le
\dim(\Im(\pi\circ\Delta))=1<2$, and hence $\gamma$ is not
surjective. Thus, there exists an irreducible component as asserted.

(4) Assume that $q_0\in C_0\cap E$. Then $q_0\notin E^{\rm sing}$ by
(1). Without loss of generality, $s_1(0)$ is a preimage of $q_0$. Assume
to the contrary that
$df_0(\Tt_{s_1(0)}(\widetilde{C}_0))=\Tt_{q_0}(E)$. Set $m:=1$, and
consider diagram \eqref{eq:diagdif}. The image of $\gamma$ belongs to
the image of $\Def^1(q_0\to E)=\Tt_{q_0}(E)\to \nN_{f_0}\otimes k(p^1)$,
which is zero. Thus, there exists an irreducible component $C\subset
C_0$ such that $-K.C\le 1$, which is a contradiction. Hence no branch of
$C_0$ is tangent to $E$. Assume now that $q_0\in C_0^{\rm
sing}$. Without loss of generality, $s_1(0)$ and $s_2(0)$ are preimages
of $q_0$. Set $m:=2$, and consider diagram \eqref{eq:diagdif}. The image
of $\gamma$ belongs to the image of
$\Tt_{q_0}(E)\to\oplus_{i=1}^2\left(\nN_{f_0}\otimes k(p^i)\right)$,
which is at most one-dimensional. Thus, $\gamma$ is not surjective, and
hence there exists an irreducible component $C\subseteq C_0$ such that
$-K.C\le |\widetilde{C}\cap \{p^1,p^2\}|\le 2$. However, $-K.C\ge 2$ by
the assumption. Hence $p^1,p^2\in \widetilde{C}$, $q_0\in C^{\rm sing}$,
and $-K.C=2$.
\end{proof}

\subsection{Conclusions and final remarks} \noindent\medskip

First, let us prove the assertion about the codimension in
Thm.\,\ref{ththreshold}: The upper bound follows easily from the fact
that the locus of equigeneric deformations in the space of all
deformations has codimension $\delta$, as explained at the very
beginning of the proof of Lem.\,\ref{prcodim}. The lower bound follows
from Prp.\,\ref{prop:dimbd} (2) and Rmk.\,\ref{rem:dimcodim} applied to
every irreducible component $W\subseteq |\lL|^\delta\setminus V$.

Second, let us prove the most difficult part of Prp.\,\ref{prNC}, namely
the nodality of a general curve in $|\lL|^\delta\setminus V$: Pick an
irreducible component $W\subseteq |\lL|^\delta\setminus V$, and let
$0\in W$ be a general closed point. Then $\dim(W)=-K.C_0+p_g(C_0)-1$ by
Thm.\,\ref{ththreshold} and Rmk.\,\ref{rem:dimcodim}. Furthermore, $C_0$
is immersed by Prp.\,~\ref{prop:dimbd} (3) and assumption (3) of
Thm.\,\ref{ththreshold}. By Prp.\,~\ref{prop:nodality} (2), if $C_0$ has
a point of multiplicity at least three, then we get a contradiction to
assumption (4), or (5), or (6), or (7), or (8) of
Prp.\,\ref{prNC}. Similarly, by Prp.\,~\ref{prop:nodality} (3), if $C_0$
has a singular point with at least two tangent branches, then we get a
contradiction to assumption (9) or (10) of Prp.\,\ref{prNC}. Thus, $C_0$
is nodal.

Third, note that Prp.\,~\ref{prop:dimbd} and Prp.\,\ref{prop:nodality}
imply few previously known results about families of curves on algebraic
surfaces such as \cite[Thm.\,2, p.\,220]{Z}, \cite[(3.1), p.\, 95]{AC},
\cite[(10.7), p.\,847]{ACG}, \cite[Prp.\,2.2,
p.\,355]{CH},\cite[Prp.\,8.1, p.\,74]{Va}, and \cite[Thm.\,2.8,
p.\,8]{Ty}.

Finally, let us mention that in positive characteristic
Prp.\,\ref{prop:dimbd} and Prp.\,\ref{prop:nodality} are no longer
true. It was shown in \cite{Ty11} that there exist $S, \lL, W$ as in the
Propositions such that: (a) for any \'etale morphism $U\to W$ the family
$C_U$ is not equinormalizable, (b) $\dim(W)=-K.C_0+p_g(C_0)-1$, and (c)
all curves $C_w$ are non-immersed, have tangent branches, and intersect
each other non-transversally. However, at least for toric surfaces $S$,
it was shown that the bound $\dim(W)\le-K.C_0+p_g(C_0)-1$ holds true in
arbitrary characteristic.

\bibliographystyle{amsplain}

\end{document}